\documentclass[12pt, a4paper, draft, twoside, reqno]{amsart}
\usepackage{amssymb}
\usepackage{amsmath}

\newtheorem{theorem}{Theorem}

\newtheorem{proposition}[theorem]{Proposition}
\newtheorem{remark}[theorem]{Remark}

\begin{document}

\title[Natural invariants and equivalence of operators]{On natural invariants and equivalence of differential operators}
\author{ Valentin Lychagin,  Valeriy Yumaguzhin}

\begin{abstract}
We give a description of the field of rational natural differential
invariants for a class of nonlinear differential operators of order $k\ge 2$
 on a smooth manifold of dimension $n\ge 2$ and show their application to the equivalence problem of such operators.
\end{abstract}

\maketitle

\section{Introduction}

This paper completes our series of publications: \cite{LY1, LY2, LY3, LY4, LY5} devoted to differential invariants and the equivalence problem for linear (and a some class of weakly nonlinear) differential operators.

The problems of this type have a long history, but the maximal number of
significant results were obtained for operators of the second order.

Riemann \cite{RB} was the first, who analysed this problem and found curvature as
an obstruction to transform differential operators of the second order to operators with constant
coefficients.

In dimension two, Laplace \cite{Lap} found "Laplace invariants" which are relative
invariants of subgroup of rescaling transformations of unknown functions and Ovsyannikov [3] found the corresponding invariants.

It is worth to note, that for the case of ordinary differential operators it
was done by Kamran and Olver \cite{KO} and for the case of linear ordinary differential equations of any order relative invariants were found by Wilczynski \cite{Wil}.

All invariants, for hyperbolic equations in dimension two with respect to the
diffeomorphism pseudogroup were found by Ibragimov \cite{Ibr}.

In papers \cite{LY1} and \cite{LY2}, we used the Levi-Civita connection,
defined by the symbols of second order operators, and the quantization,
associated with this connection, found the field of rational differential
invariants for both linear and a class of nonlinear second order
differential equations. It is also allowed us to solve the equivalence
problem.

The case of differential operators of order $k=3$ on two-dimensional
manifolds was elaborated in \cite{LY4}, where we used the Wagner connection
instead of the Levi-Civita and associated with this connection quantization.

All these results were based on the existence of connections associated with
operators in the natural way. So, this method we applied in \cite{LY3} to
differential operators of order $k\geq 3,$ in arbitrary dimension, that have
constant type, i.e. such operators that the $\mathrm{GL}$-orbits of their
symbols do not depend on points of the base manifolds.

In this paper, we removed all above restrictions on the symbols and
considered arbitrary differential operators of order $k\geq 3$ in general
position, i.e. such operators that $\mathrm{GL}$-orbits of their symbols
are regular.

This paper is organized in the following way. At first, we are recalling the
basic results on natural invariants of differential operators and introduce
a class of nonlinear operators, where nonlinearity has defined by a finite
extension of the field of rational functions in one variable. Then we pay
attention to the fact that $\mathrm{GL}$-invariants of symbols are natural
differential invariants of the zero order for differential operator
themselves. The finding of $\mathrm{GL}$-invariants of $n$-ary forms is the
basic problem of invariant theory and goes back to seminal Hilbert's paper 
\cite{Hil}.

 Remark, that nowadays there are various algebraic methods for
algorithmic finding of such invariants (\cite{Dec, Pop}).

We discuss here only classical methods, based on transfectants (\cite{Olv, Br}). 

The crucial result here says that it is enough to have $n$
(=dimension of the base manifold) of natural invariants in general position to describe the complete field of rational differential invariants of symbols and of linear differential operators them selves. 

The application of the Rosenlicht theorem \cite{Ros} shows that this field
enough to describe regular orbits of the actions of the diffeomorphism
pseudogroup on the jets of differential operators. 

Moreover, it allows us to construct local models of linear differential
operators. 

Then, we illustrate the method in two and three-dimensional cases.

In the final part of the paper, we translate this approach to a class of
nonlinear operators.

\section{Preliminaries}

In this section, we collected some constructions and notions from jet
geometry (\cite{KL},\cite{KLV}) , that are important for us in this paper.

Let $M$ be a connected, smooth manifold, and let $\pi :E\left( \pi \right)
\rightarrow M$ be a vector bundle over $M.$

We denote by $C^{\infty }\left( \pi \right)$ the $C^{\infty }(M)$-module of smooth sections of 
$\pi$.

By $\pi _{k}:J^{k}\left( \pi \right) \rightarrow M$ we denote the bundles of 
$k$-jets of smooth sections of the bundle $\pi $, and by $\pi
_{k,l}:J^{k}\left( \pi \right) \rightarrow J^{l}\left( \pi \right) $ the
reductions of $k$-jets to $l$-jets, when $k>l$.

The following vector bundles and their jet bundles will be important for us:

\begin{itemize}
\item The tangent bundle $\tau :TM\rightarrow M$ and its symmetric powers $%
\tau _{k}:S^{k}TM\rightarrow M$. 

Their modules of sections we denote by 
$$
\Sigma _{k}(M) =C^{\infty}(\tau _k) 
$$ 
and elements of these modules are $k$-symmetric vector fields.

\item The bundles $\psi _{k}\colon\mathit{Diff}_{k}(M)\rightarrow M$ of
linear scalar differential operators of order $\leq k$. \newline
Their modules of sections $C^{\infty }\left( \psi _{k}\right) $ can be
identified with modules $\mathrm{Diff}_{k}(M)$ of linear
differential operators on $M,$ having order $\leq k$. \newline
The section, that corresponds to an operator $A\in \mathrm{Diff}_{k}\left(
M\right) ,$ we will denote by $s_{A}.$

\item We have the following exact sequences of the modules 
\begin{equation}
0\rightarrow \mathrm{Diff}_{k-1}\left( M\right) \rightarrow 
\mathrm{Diff}_{k}\left( M\right) \overset{\mathrm{smbl}}{%
\rightarrow }\Sigma _{k}\left( M\right) \rightarrow 0,  \label{exact seq}
\end{equation}%
where the $\mathrm{smbl}$-map sends operators $A\in \mathrm{Diff}%
_{k}\left( M\right) $ to their symbols $\sigma_A\in\Sigma _k(M)$.
\end{itemize}

Remark, that all these bundles and the correspondent jet bundles are natural, i.e. the action on $M$ of pseudogroup $\mathcal{G}\left( M\right) $
of local diffeomorphisms$,$ lifts to local automorphisms of the bundles 
$\psi_k$, $\tau_k$ and by prolongations lifts to local automorphisms 
of the their $l$-jet bundles $\psi_{k,l}, \tau_{k,l}$.

Functions on $J^{l}\left( \psi _{k}\right) $ or $J^{l}\left( \tau
_{k}\right) $ that are invariant with respect to these action are called 
\textit{natural }$l$-\textit{invariants} of differential operators or $k$%
-symmetric vector fields.

It is worth to note, that the $\mathcal{G}\left( M\right) -$action is
transitive on $M,$ and, therefore, the natural invariants are completely
defined by their values on the fibres $J_{b}^{l}\left( \psi _{k}\right) $ or 
$J_{b}^{l}\left( \tau _{k}\right) $ at a fixed basic point $b\in M.$

Denote by $\mathbf{G}_{r}$ the Lie group of $r$-jets of local
diffeomorphisms of $M,$ that fix the base point $b.$ Then the natural
invariants coincide with $\mathbf{G}_{r}$-invariants of the linear $\mathbf{G%
}_{k+l}$- action on vector space $J_{b}^{l}\left( \psi _{k}\right) $ or the
linear $\mathbf{G}_{l+1}$- action on vector space $J_{b}^{l}\left( \tau
_{k}\right) .$

Thus, we have algebraic action of the algebraic group, and in this case, the
Rosenlicht theorem \cite{Ros} states that the field of the rational natural
invariants separate regular orbits and that the codimension of a regular
orbit coincides with the transcendence degree of the field of rational
natural invariants.

In what follows, we will denote by $\mathcal{F}_{k, l}$, or simpler $\mathcal{%
F}_{l}$ ,when the order of operators fixed, the field of rational natural $l$%
-invariants of linear differential operators of order $\leq k.$

The similar notation, $\mathcal{F}_{k,l}^{\sigma }$ will be used for the
field of rational natural $l$-invariants of the $k$-symmetric vector fields.

Remark, that all morphisms in (\ref{exact seq}) are natural, i.e. commute
with the $\mathcal{G}\left( M\right) -$actions, and therefore we have the
following statement.

\begin{proposition}
The symbol mapping $\mathrm{smbl}$ establishes embedding of fields: 
\begin{equation*}
\mathcal{F}_{k,l}^{\sigma }\subset \mathcal{F}_{k,l}.
\end{equation*}
\end{proposition}

\section{Invariant universal constructions}

Follow the construction of the universal differential $1$-form (so-called
Liouville form ) on the cotangent bundle $\tau ^{\ast }:T^{\ast
}M\rightarrow M,$ we define (see \cite{LY3}) universal differential operators on the
jet bundles of the differential operators bundles $\psi _{k},$ as well as on
jet bundles of the bundles $\tau _{k}.$

Remind, that the Liouville form is a differential 1-form $\rho $ on the
total space $T^{\ast }M$ of the cotangent bundle $\tau ^{\ast },$ $\rho \in
\Omega ^{1}\left( T^{\ast }M\right) ,$ such that 
\begin{equation*}
s_{\theta }^{\ast }\left( \rho \right) =\theta ,
\end{equation*}%
for any differential form, $\theta \in \Omega ^{1}\left( M\right) ,$ on the
base manifold $M.$

Here, $s_{\theta }:M\rightarrow T^{\ast }M$ is the section of the cotangent
bundle that correspond to the form $\theta .$

This construction valid word by word to the bundles of symmetric k-forms: $%
\tau ^{k}:S^{k}T^{\ast }M\rightarrow M$ and define the k-analogue, $\rho
_{k}\in \Sigma ^{k}\left( S^{k}T^{\ast }M\right) ,$ of the Liouville form,
but it is not valid for bundles $\psi _{k}$ of differential operators or
bundles $\tau _{k}$ of symmetric vector fields.

To apply this construction to the differential operator bundles, we will
consider, at first, bundles of infinite order jets, say $\pi ^{\cdot
}:J^{\infty }(\pi)\rightarrow M$. 

Then, any vector field $X$ on manifold $M$ defines it \textit{total lift}
$\widehat{X}$ , i.e. derivation $\widehat{X}\colon C^{\infty}\big(J^{\infty}(\pi)
\big) \rightarrow C^{\infty }\big(J^{\infty}(\pi)\big)$, i.e. sequence of derivations 
$\widehat{X}:C^{\infty }\left( J^{l}\left(
\pi \right) \right) \rightarrow C^{\infty }\left( J^{l+1}\left( \pi \right)
\right) ,$ $l\geq 0.$

Namely, if we consider functions on the jet bundle $J^{l}\left( \pi \right)
$, say $f\in C^{\infty }\left( J^{l}\left( \pi \right) \right)$, as
nonlinear differential operators $\Delta _{f}:C^{\infty }\left( \pi \right)
\rightarrow C^{\infty }\left( M\right) $ of order $\leq l,$ that acts as
follows%
\begin{equation*}
\Delta _{f}\left( h\right) =\left( j_{l}\left( h\right) \right) ^{\ast
}\left( f\right) ,
\end{equation*}%
where $j_{l}\left( h\right) :M\rightarrow J^{l}\left( \pi \right) $ is the
section that corresponds to the $l$-th jet of the section $h\in C^{\infty
}\left( \pi \right) .$

In this interpretation, operator $\Delta _{\widehat{X}\left( f\right) }$
that correspond to the function $\widehat{X}\left( f\right) $ is the
composition $X\circ \Delta _{f}$, i.e.%
\begin{eqnarray*}
\Delta _{\widehat{X}\left( f\right) } &=&X\circ \Delta _{f}, \\
\Delta _{\widehat{X}\left( f\right) }\left( h\right) &=&X\left( \left(
j_{l}\left( h\right) \right) ^{\ast }\left( f\right) \right).
\end{eqnarray*}%
The similar total lifting construction works also for differential operators.

Indeed, if $A\in\mathrm{Diff}_{k}(M)$, then the lift $%
\widehat{A}\colon C^{\infty}\big(\! J^{\infty}(\pi)\big)
\rightarrow C^{\infty}\big(\! J^{\infty}(\pi)\big)$, where 
$\widehat{A}:C^{\infty }\left( J^{l}\left( \pi \right) \right) \rightarrow
C^{\infty }\left( J^{l+k}\left( \pi \right) \right) ,$ $l\geq 0,$ acts as
above,%
\begin{equation*}
\Delta _{\widehat{A}\left( f\right) }\left( h\right) =A\left( \left(
j_{l}\left( h\right) \right) ^{\ast }\left( f\right) \right) .
\end{equation*}%
In the case, when $\pi =\psi _{k},$ we define universal differential
operator 
\begin{equation*}
\square :C^{\infty }\left( J^{\infty }\left( \psi _{k}\right) \right)
\rightarrow C^{\infty }\left( J^{\infty }\left( \psi _{k}\right) \right) ,
\end{equation*}%
where $\square :C^{\infty }\left( J^{l}\left( \psi _{k}\right) \right)
\rightarrow C^{\infty }\left( J^{l+k}\left( \psi _{k}\right) \right) ,$ as
follows.

Take a function $f\in C^{\infty }\left( J^{l}\left( \psi _{k}\right) \right)
,$ then we define the value of function $\square \left( f\right) \in
C^{\infty }\left( J^{l+k}\left( \psi _{k}\right) \right) $ at point $%
[A]_{a}^{k+l}\in J_{a}^{l+k}\left( \psi _{k}\right) ,$ that equals to $%
\left( k+l\right) $-jet of an operator $A\in \mathrm{Diff}_{k}\left(
M\right) $ at a point $a\in M,$ as value of $\widehat{A}\left( f\right) $ at
the point $a\in M.$

In the similar way, we define an universal $k$-symmetric vector field $\sigma
_{k}\in \Sigma _{k}\left( J^{\infty }\left( \tau _{k}\right) \right)$.

Namely, we define a total lift $\widehat{\theta }$ of a k-symmetric vector
field $\theta =X_{1}\cdot \ldots \cdot X_{k},$ where $X_{i}$ are vector
fields on $M$ and $\cdot $ stands for the symmetric product of the vector
fields, as follows: $\widehat{\theta }=\widehat{X}_{1}\cdot \ldots \cdot 
\widehat{X}_{k}.$ Then, the value of $\sigma _{k}$ at point $\theta _{a}\in
S^{k}T_{a}M$ equals to $\widehat{\theta }_{a}.$

Remark, that, by the construction, all these operators $\square ,$ as well
as tensors $\sigma _{k},$ are invariants of the diffeomorphism pseudogroup.

\begin{description}
\item[Coordinates] In the bundles $\psi _{k}$ we will take local coordinates 
$\left( x_{1}, \ldots, x_{n},u_{\alpha }\right) ,$ where $\left(
x_{1}, \ldots, x_{n}\right) $ are local coordinates on $M,$ $\alpha $ are multi
indices of lengths $\left\vert \alpha \right\vert \leq k,$ and $u_{\alpha
}\left( A\right) =A_{\alpha },$ if $A=\sum_{\alpha }A_{\alpha
}\partial ^{\alpha }$ in the coordinates $\left( x_{1}, \ldots, x_{n}\right) .$ 
\newline
The same coordinates $\left( x_{1}, \ldots ,x_{n},u_{\alpha }\right) ,\left\vert
\alpha \right\vert =k,$ we will use in the bundles $\tau _{k}.$\newline
In the jet bundles $\psi _{k,l}$ and $\tau _{k,l}$ we will use the standard
jet coordinates $\left( x_{1}, \ldots ,x_{n},u_{\alpha ,\beta }\right) ,$ where
multi indices $\beta $ stand for the differential orders.\newline
Then, the universal differential operator has the form:%
\begin{equation*}
\square =\sum\limits_{\left\vert \alpha \right\vert \leq k}u_{\alpha }\frac{%
d^{\left\vert \alpha \right\vert }}{dx^{\alpha }},
\end{equation*}%
where $\displaystyle\frac{d}{dx_{i}}$ are the total derivatives.\newline
\newline
Respectively, the universal $k-$symmetric vector field $\sigma _{k}$ has the
form:%
\begin{equation*}
\sigma _{k}=\sum\limits_{\left\vert \alpha \right\vert =k}u_{\alpha }\left( 
\frac{d}{dx_{1}}\right) ^{\alpha _{1}}\cdot \ldots \cdot \left( \frac{d}{%
dx_{n}}\right) ^{\alpha _{n}}.
\end{equation*}
\end{description}

\section{The principle of n-invariants and equivalence of differential
operators}

We say that natural invariants $I_{1}, \ldots, I_{n}\in C^{\infty }\left(
J^{l}\left( \psi _{k}\right) \right)$, $n=\dim M,$ are in general position
in a domain $\mathcal{O\subset }$ $J^{l}\left( \psi _{k}\right) $ if 
\begin{equation*}
\widehat{d}I_{1}\wedge \cdots \wedge \widehat{d}I_{n}\neq 0
\end{equation*}%
in this domain.

Here $\widehat{d}$ is the total lift of the de Rham differential, $\Delta _{%
\widehat{d}I}=d\circ \Delta _{I},$ and in local coordinates it has the form: 
\begin{equation*}
\widehat{d}I=\sum\limits_{i=1}^{n}\frac{dI}{dx_{i}}dx_{i}.
\end{equation*}%
\ 

Let $A\in \mathrm{Diff}_{k}\left( M\right) ,$ $\mathcal{O}^{\prime }%
\mathcal{\subset }M,$ be such differential operator and domain in $M,$ that $%
S_{A}\left( \mathcal{O}^{\prime }\right) \subset \mathcal{O}.$

Then functions $x_{i}=I_{i}\left( A\right) ,$ $1\leq i\leq n,$ where $%
I\left( A\right) =S_{A}^{\ast }\left( I\right) $ is the restriction of
invariant $I$ on the graph of the section $S_{A},$ are coordinates in domain 
$\mathcal{O}^{\prime }.$

In this coordinates differential operator $A$ has the form%
\begin{equation}
A=\sum\limits_{\left\vert \alpha \right\vert \leq k}A_{\alpha }\partial
^{a},
\end{equation}%
where 
\begin{equation*}
A_{\alpha }=\frac{1}{\alpha !}A\left( x^{\alpha }\right) .
\end{equation*}%
We would like to avoid of using of these coordinates $x_{i},$ and will write
down this construction in terms of invariants only.

Namely, remind that the Tresse derivatives, we write down them as 
\begin{equation*}
\frac{df}{dI_{i}},
\end{equation*}%
where $f\in C^{\infty }\left( J^{\infty }\left( \psi _{k}\right) \right) ,$
are defined by the following relation :%
\begin{equation*}
\widehat{d}f=\sum\limits_{i=1}^{n}\frac{df}{dI_{i}}\widehat{d}I_{i},
\end{equation*}%
and they well defined in domains, where invariants $I_{i}$ are in general
position.

Then, the above functions $A_{\alpha }$ are restrictions of functions 
\begin{equation*}
J_{\alpha }=\frac{1}{\alpha !}\square \left( I^{\alpha }\right) .
\end{equation*}%
Now, the restriction of this formula on sections $S_A$ for all operators
$A$, such that $S_{A}\left(\mathcal{O}^{\prime }\right) \subset \mathcal{O}$,
gives us the above representation $A$ in the local coordinates.

Moreover, operator $\square $ commutes with $\mathcal{G}\left( M\right) $%
-action and therefore functions $J_{\alpha }$ are natural invariants.

Summarize, we get the following results.

\begin{theorem}[The n-invariants principle]
Let natural invariants\\ $I_{1}$, $\ldots$, $I_{n}$ be in general position in a domain $%
\mathcal{O\subset }J_{b}^{\infty }\left( \psi _{k}\right) $. Then in this
domain all natural invariants of $k$-th order linear differential operators
are rational functions of invariants $J_{\alpha }=\displaystyle\frac{1}{\alpha !}\square
\left( I^{\alpha }\right), \left\vert \alpha \right\vert \leq k$, and their
Tresse derivatives.
\end{theorem}

This theorem allow us to construct models (or normal forms) of differential
operators.

Consider space $\Phi _{k}=\mathbb{R}^{n}\times\mathbb{R}^{\binom{n+k}{k}}$
with coordinates $\left( y_{1}, \ldots, y_{n}, Y_{\alpha }\right)$, where $\left\vert \alpha
\right\vert \leq k $. Then any differential operator $A\in \mathrm{Diff}_{k}(M)$ in a domain 
$\mathcal{O}\subset M$ defines a map%
$$
\phi _{A}\colon\mathcal{O}\rightarrow\Phi_{k},\quad
\phi _{A}\colon y\mapsto\big( y_{1}=I_{1}\left(A\right), \ldots, 
y_{n}=I_{n}\left(A\right), Y_{\alpha }=J_{\alpha }\left(A\right)\big).
$$
We call a pair $\left(A, \mathcal{O}\right)$ \textit{adjusted} if functions 
$\big( I_{1}\left(A\right),\ldots, I_{n}\left(A\right) \big)$ are
coordinates in the domain $\mathcal{O}$.

Then $n$-dimentional submanifold $\Sigma _{A}=\phi _{A}\left( \mathcal{O}%
\right) \subset\Phi_{k}$ we call \textit{model of the operator} in the domain.

\begin{theorem}
Let $\left( A_{1},\mathcal{O}_{1}\right) $ and $\left( A_{2},\mathcal{O}%
_{2}\right) $ , $A_{i}\in \mathrm{Diff}_{k}\left( M\right) ,\mathcal{O}%
_{i}\subset M,$ are adjusted pairs. Then there exist a diffeomorphism $F:%
\mathcal{O}_{1}\rightarrow \mathcal{O}_{2}$ such that $F_{\ast }\left(
A_{1}\right) =A_{2}$ if and only if their models equal, $\Sigma _{A_{1}}=$ $%
\Sigma _{A_{2}}.$
\end{theorem}

\begin{remark}
The same type theorem valid for $k-$symmetric vector fields if we take 
\begin{equation*}
J_{\alpha }=\left\langle \sigma _{k},\widehat{d}h^{\alpha }\right\rangle ,
\end{equation*}%
where $,\widehat{d}h^{\alpha }$ is the symmetric power and $\left\vert
\alpha \right\vert =k.$
\end{remark}

\section{Transfectants, $\mathrm{GL}$-invariants of $n$-ary forms and zero
order natural invariants}

Keeping in mind algebraic invariants (i.e. differential invariants of zero
order) of symbols, we discuss here $\mathrm{GL}$-invariants of $n$-ary
forms.

It is worth to note that these invariants are rational functions that, in
turn, are ratios of $\mathrm{SL}$-invariants.

Remark, that the Rosenlicht theorem states that the transcendence degree of
the field of rational invariants equals to the codimension of the regular
orbit.

In the case, when the codimension greater or equal to $n=\dim M,$ it means
that, in general case, differential invariants of zero order enough to
construct models of differential operators.

So, we consider a vector space $V$ (equals to $T_{b}^{\ast }M\ $in the
previous notations) equipped with volume n-vector $\varpi =n!e_{1}\wedge
\cdots \wedge e_{n}$ in a basis $\left\{ e_{1},\cdots ,e_{n}\right\} $ of 
$V$.

Let $\mathbf{S}^{\cdot }=\oplus _{k\geq 0}S^{k}V^{\ast }$ be the algebra of
polynomials on $V$ and $\mathrm{SL}\left( V\right) $ be the Lie group of
linear transformations, preserving $\varpi .$

Write down n-vector $\varpi $ in the original form 
\begin{equation*}
\varpi =\sum\limits_{\sigma \in S_{n}}\left( -1\right) ^{\sigma }e_{\sigma
\left( 1\right) }\otimes \cdots \otimes e_{\sigma \left( n\right) }\in
V^{\otimes n},
\end{equation*}%
denote $\partial _{i}$ the derivations of degree $\left( -1\right) $ in
the algebra $\mathbf{S}^{\cdot }$ along vectors $e_{i}$, $\partial _{i}\left(
S^{k}V^{\ast }\right) \subset S^{k-1}V^{\ast }$, and represent $n$-vector $%
\varpi $ as the operator%
\begin{equation*}
\nabla =\sum\limits_{\sigma \in S_{n}}\left( -1\right) ^{\sigma }\partial
_{\sigma \left( 1\right) }\otimes \cdots \otimes \partial _{\sigma \left(
n\right) }:\left( \mathbf{S}^{\cdot }\right) ^{\otimes n}\rightarrow \left( 
\mathbf{S}^{\cdot }\right) ^{\otimes n},
\end{equation*}%
where 
\begin{equation*}
\left( \partial _{\sigma \left( 1\right) }\otimes \cdots \otimes \partial
_{\sigma \left( n\right) }\right) \left( f_{1}\otimes \cdots \otimes
f_{n}\right) =\partial _{\sigma \left( 1\right) }\left( f_{1}\right) \otimes
\cdots \otimes \partial _{\sigma \left( n\right) }\left( f_{n}\right) .
\end{equation*}

Define now \textit{transfectant }of order $l\geq 0,$ as operator 
\begin{equation*}
T_{l}=\mu \circ \nabla ^{l}:\left( \mathbf{S}^{\cdot }\right) ^{\otimes
n}\rightarrow \mathbf{S}^{\cdot },
\end{equation*}%
where $\mu :$ $\left( \mathbf{S}^{\cdot }\right) ^{\otimes n}\rightarrow 
\mathbf{S}^{\cdot }$ is the multuplication map: $\mu \left( f_{1}\otimes
\cdots \otimes f_{n}\right) =f_{1}\cdots f_{n}.$

We will denote the transfectant as follows 
$$
T_{l}\left( f_{1}\otimes \cdots\otimes f_{n}\right) =\left\{ f_{1},\cdots ,f_{n}\right\} _{l}.
$$

Remark, that both operators $\nabla $ and $\mu $ are $\mathrm{SL}\left(
V\right) $-invariant and therefore the transfectant operator is also $%
\mathrm{SL}\left( V\right) $-invariant\textit{.}

Moreover, the transfectants $\left\{ f_{1},\cdots ,f_{n}\right\} _{l}$ are

\begin{itemize}
\item skew symmetric%
\begin{equation*}
\left\{ f_{\sigma \left( 1\right) },\cdots ,f_{\sigma (n)}\right\}
_{l}=\left( -1\right) ^{l\sigma }\left\{ f_{1},\cdots ,f_{n}\right\}.
\end{equation*}

\item degree $-nl:$ 
\begin{eqnarray*}
T_{l} &:&S^{r_{1}}V^{\ast }\otimes S^{r_{2}}V^{\ast }\otimes \cdots \otimes
S^{r_{n}}V^{\ast }\rightarrow S^{R}V^{\ast }, \\
R &=&\sum\limits_{i=1}^{n}r_{i}-nl,\quad\text{if}\quad r_i\ge l. 
\end{eqnarray*}

\item Transfectants $J\left( f\right) =\left\{ f,\cdots ,f\right\} _{\deg
f}, $are $\mathrm{SL}\left( V\right) -$invariants for any n-ary forms.

Remark, that $J\left( f\right) =0$ if $\deg f$ - odd.

\item The explicit formulae for the transfectant are the following%
\begin{eqnarray*}
&&\left\{ f_{1},\cdots ,f_{n}\right\} _{l}= \\
&=&\sum\limits_{\sigma \in
A_{n}}\sum\limits_{k=0}^{l}\sum\limits_{\left\vert k_{\sigma }\right\vert
=l-k}\sum\limits_{\left\vert l_{\sigma }\right\vert =k-l}\left( -1\right)
^{k}\binom{l}{k}\binom{l-k}{k_{\sigma }}\binom{k}{l_{\sigma }}\mu (\partial
_{\sigma }^{k_{\sigma }}\partial _{\tau \sigma }^{l_{\tau\sigma }}\left(
f_{1}\otimes \cdots \otimes f_{n}\right) ),
\end{eqnarray*}%
where $k_{\sigma }=\left( k_{\sigma\left(1\right)}, \ldots, k_{\sigma\left(n\right)}\right) ,\sigma=\left(\sigma\left(1\right), \ldots, \sigma\left(n\right)\right) $ are multi indices, and $\tau \in S_{n}$ is a fixed odd permutation,%
\begin{equation*}
\partial _{\sigma }^{k_{\sigma }}=\partial _{\sigma \left( 1\right)
}^{k_{\sigma \left( 1\right) }}\otimes \cdots \otimes \partial _{\sigma
\left( n\right) }^{k_{\sigma \left( n\right) }}.
\end{equation*}

\item Thus, for the two-dimensional case, we have%
\begin{equation*}
\left\{ f_{1},f_{2}\right\} _{l}=\sum\limits_{k=0}^{l}\left( -1\right) ^{k}%
\frac{\partial ^{l}f_{1}}{\partial x_{1}^{l-k}\partial x_{2}^{k}}\frac{%
\partial ^{l}f_{2}}{\partial x_{1}^{k}\partial x_{2}^{l-k}}.
\end{equation*}
\end{itemize}

\subsection{Two-dimensional case}

Let $\dim V=2,$ then regular $\mathrm{SL}_{2}$-orbits in $S^{k}V^{\ast
},k\geq 3,$ have codimension $k-2,$ and therefore, regular $\mathrm{GL}_{2}$%
-orbits have codimension $k-3.$

Thus, begining with order $k=5,$ we have enough zero order invariants to
apply the 2-invariants principle.

The case, $k=3,$ we analysed in \cite{LY3} by using the Wagner connection and
associated quantization.

In the case, $k=4,$ we have two independent $\mathrm{SL}_{2}$-invariants: 
\begin{eqnarray*}
J_{2}\left( P\right)  &=&\frac{\left\{ P,P\right\} _{4}}{%
2^{7}3^{2}}, \\
J_{3}\left( P\right)  &=&\frac{\left\{ \left\{ P,P\right\}
_{2},P\right\} _{4}}{2^{11}3^{5}}.
\end{eqnarray*}%
These invariants have the form%
\begin{eqnarray*}
J_{2} &=&p_{0}p_{4}-4p_{1}p_{3}+3p_{2}^{2}, \\
J_{3}
&=&p_{0}p_{2}p_{4}-p_{0}p_{3}^{2}-p_{1}^{2}p_{4}+2p_{1}p_{2}p_{3}-p_{2}^{3},
\end{eqnarray*}%
when%
\begin{equation*}
P=p_{4}x^{4}+4p_{3}x^{3}y+6p_{2}x^{2}y^{2}+4p_{1}xy^{3}+p_{0}y^{4}.
\end{equation*}%
Remark, that these invariants have low degrees than the commonly used
discriminant of $P,$ that has degree 6 and equals $256J_{2}^{3}-6912J_{3}^{2}
$.

We get also $\mathrm{GL}_{2}$ -invariant 
\begin{equation*}
J=\frac{J_{2}^{3}}{J_{3}^{2}},
\end{equation*}%
which separate regular orbits: $J\neq 27.$

To get additional natural invariants for differential operators of order 4
and apply the 2-invariants principle, we can use invariants $I_{1}=\square
\left( 1\right) ,I_{2}=J$ as well as $I_{\alpha }=\square \left(
I_{1}^{\alpha _{1}}I_{2}^{\alpha _{2}}\right) ,$ for integers $\alpha
_{1},\alpha _{2}.$

In the case $k=5,$ we have 3 independent $\mathrm{SL}_{2}$-invariants:%
\begin{eqnarray*}
J_{4}\left( P\right)  &=&\left\{ c_{21},c_{21}\right\} _{2}, \\
J_{8}\left( P\right)  &=&\left\{ c_{21},c_{22}\right\} _{2}, \\
J_{12}\left( P\right)  &=&\left\{ c_{22},c_{22}\right\} _{2}
\end{eqnarray*}%
where%
\begin{equation*}
c_{21}=\left\{ P,P\right\} _{4},\ c_{3}=\left\{ P,c_{21}\right\} _{2},\
c_{22}=\left\{ c_{3},c_{3}\right\} _{2}.
\end{equation*}%
They give us two independent $\mathrm{GL}_{2}$ -invariants%
\begin{equation*}
I_{1}=\frac{J_{8}\left( P\right) }{J_{4}\left( P\right) ^{2}},I_{2}=\frac{%
J_{12}\left( P\right) }{J_{4}\left( P\right) J_{8}\left( P\right) }.
\end{equation*}%
Therefore, to apply the 2-invariants principle, we can use two of the
following natural invariants:%
\begin{equation*}
I_{1},I_{2},I_{0}=\square \left( 1\right) ,I_{\alpha }=\square \left(
I_{0}^{\alpha _{0}}I_{1}^{\alpha _{1}}I_{2}^{\alpha _{2}}\right) ,
\end{equation*}%
where $\alpha _{0}.\alpha _{1},\alpha _{2}$ are integers.

\subsection{Three-dimensional case}

Because the case $k=2$ is exceptional and was completely elaborated in \cite%
{LY2}, we consider here the case\\ $k=n=3$ only.

The codimension of regular $\mathrm{SL}_{3}$-orbits in $S^3V^*$ equals to 2, and
therefore to separate regular orbits we need two independent $\mathrm{SL}%
_{3}$-invariants.

First we take the $\mathrm{SL}_{3}$-invariant%
\begin{equation*}
J_{1}\left( P\right) =\left\{ P^{2},P^{2},P^{2}\right\} _{6}.
\end{equation*}%
The second \ $\mathrm{SL}_{3}$-invariant we choose the following 
\begin{equation*}
J_{2}\left( P\right) =\left\{ P,c_{1},c_{2}\right\} _{3},
\end{equation*}%
where%
\begin{equation*}
c_{1}=\left\{ P,P,P\right\} _{2},\quad c_{2}=\left\{ P,P,c_{1}\right\} _{2}.
\end{equation*}%
This gives us the $\mathrm{GL}_{3}$ -invariant%
\begin{equation*}
J=\frac{J_{2}^{2}}{J_{1}^{3}}.
\end{equation*}%
Then, as above, we can use $\mathrm{GL}_{3}$ -invariants in the form 
$I_{1}=J$, $I_{2}=\square(1)$, 
$I_{\alpha }=\square \left(I_{1}^{\alpha _{1}}I_{2}^{\alpha _{2}}\right)$ to choose three of them and apply the 3-invariants principle.

\section{$\mathbb{F}$-nonlinear operators}

Firstly, we fix the types of non-linearities that we are going to consider.

Let $u:J^{0}\left( M\right) =M\times \mathbb{R\rightarrow }\mathbb{R}$ be
the fibrewise coordinate on the zero order jet bundle.

Let $\mathbb{Q}\left( u\right) $ be the field of rational functions in $u$
and $\mathbb{F}\supset \mathbb{Q}\left( u\right) $ be a finite field
extension.

We denote by $\mathrm{Diff}_{k}\left( M,0\right) \subset \mathrm{Diff}%
_{k}\left( J^{0}\left( M\right) \right) $ the module of linear differential
operators on manifold $J^{0}\left( M\right)$, that have order $\leq k$ and
commuting with multiplication by $u$.

In other words, elements of $\mathrm{Diff}_{k}\left( M,0\right) $ are
differential operators $A$ that in local coordinates $\left(
x_{1},...,x_{n},u\right) ,$ where $\left( x_{1},...,x_{n}\right) $ are local
coordinates on $M,$ have the following form $\ $%
\begin{equation*}
A=\sum\limits_{\left\vert \sigma \right\vert \leq k}a_{\sigma }\left(
x,u\right) \frac{\partial ^{\left\vert \sigma \right\vert }}{\partial
x^{\sigma }}.
\end{equation*}
We say that this operator is of type $\mathbb{F}$ if functions $a_{\sigma
}\left( x,u\right) $ are smooth in $x$ and 
\begin{equation*}
a_{\sigma }\left( x,u\right) \in \mathbb{F},
\end{equation*}%
when $x$ is fixed.

The module of this type operators we denote by $\mathrm{Diff}_{k}\left( M,%
\mathbb{F}\right) .$

Respectively, we denote by 
\begin{equation*}
\widetilde{\psi }_{k}:\mathit{Diff}_{k}\left( M,\mathbb{F}\right) \rightarrow
J^{0}\left( M\right) 
\end{equation*}%
the corresponding vector bundles, and by $S_{A}:J^{0}\left( M\right)
\rightarrow\mathit{Diff}_{k}\left( M,\mathbb{F}\right) $ we denote sections of this
bundle that correspond to operators $A\in \mathrm{Diff}_{k}\left( M,\mathbb{%
F}\right)$.

Any such operator $A\in \mathrm{Diff}_{k}\left( M,\mathbb{F}\right) $ and
function $f\in C^{\infty }\left( M\right) $ defines a linear operator $%
A_{f}\in \mathrm{Diff}_{k}\left( M\right)$, where $A_{f}\left( g\right), 
g\in C^{\infty }\left( M\right)$ is the restriction function $A\left(
g\right) \in C^{\infty }\left( J^{0}M\right) $ on the graph $\Gamma
_{f}\subset J^{0}M$ of the function $f$.

Moreover, such operators $A$ define nonlinear differential operators $%
A_{w}\in \mathrm{diff}_{k}\left( M,\mathbb{F}\right) $ on the manifold $M,$
where%
\begin{equation*}
A_{w}\left( f\right) =A_{f}\left( f\right) ,
\end{equation*}%
for $f\in C^{\infty }\left( M\right) .$

Here, we denoted by $\mathrm{diff}_{k}\left( M,\mathbb{F}\right) $ the
space of scalar nonlinear differential operators on the manifold $M,$ having
order $\leq k$ and nonlinearity of $\mathbb{F}$-type.

In the local coordinates $\left( x_{1},...,x_{n},u\right) $ these operators
have the following forms: 
\begin{eqnarray*}
A_{f} &=&\sum_{\left\vert \sigma \right\vert \leq k}a_{\sigma
}\left( x,f\right) \frac{\partial ^{\left\vert \sigma \right\vert }}{%
\partial x^{\sigma }}, \\
A_{w}\left( f\right)  &=&\sum_{\left\vert \sigma \right\vert \leq
k}a_{\sigma }\left( x,f\right) \frac{\partial ^{\left\vert \sigma
\right\vert }f}{\partial x^{\sigma }}.
\end{eqnarray*}

\section{Natural invariants of $\mathbb{F}$\textit{-}nonlinear differential
operators}

The diffeomorphisn pseudogroup $\mathcal{G}\left( M\right) $ acts by the
prolongations $\phi ^{\left( k\right) },$ $\phi \in \mathcal{G}\left(
M\right) ,$ in the bundles of $k$-jets $J^{k}\left( M\right) ,$ and
therefore by diffeomorphisms $\phi ^{\left( 0\right) }:\left( x,u\right)
\rightarrow \left( \phi \left( x\right) ,u\right) $ acts on $J^{0}\left(
M\right)$. 

Thereby, we get a $\mathcal{G}\left( M\right) $-action in the
bundles $\varkappa _{l}:J^{l}\left( \widetilde{\psi }_{k}\right) \rightarrow
J^{0}\left( M\right) $ of $l$-jet of operators in $\mathrm{Diff}_{k}\left(
M,\mathbb{F}\right) .$

This $\mathcal{G}\left( M\right) $-action is not transitive on $J^{0}\left(
M\right) $ and therefore we will consider prolonged bundles $\widetilde{%
\varkappa }_{l}:J^{l}\left( \widetilde{\psi }_{k}\right) \rightarrow
J^{0}\left( M\right) \rightarrow M.$ As before, we pick a basic point $b\in
M $. Then any $\mathcal{G}\left( M\right) -$invariant function on $%
J^{l}\left( \widetilde{\psi }_{k}\right) $ is completely defined by their
values on the fibre $\widetilde{\varkappa }_{l}^{-1}\left( b\right) .$

In this context by natural invariants of $\mathbb{F}$\textit{-} differential
operators we mean functions from $\mathbb{Q}\left( \widetilde{\varkappa }%
_{l}\right) \otimes \mathbb{F},$ where $\mathbb{Q}\left( \widetilde{%
\varkappa }_{l}\right) $ consist of functions on $J^{l}\left( \widetilde{%
\psi }_{k}\right) ,$ that are rational along the fibres $\widetilde{%
\varkappa }_{l}.$

Taking their values on the fibre $\widetilde{\varkappa }_{l}^{-1}\left(
b\right) ,$ we get a field $\mathbf{Q}_{l},$ where $\mathcal{G}\left(
M\right) -$ invariants form a subfield $\mathbf{F}_{l}\subset \mathbf{Q}_{l}.$

It is easy to check that mappings 
\begin{eqnarray*}
\Theta _{k} &:&\mathrm{Diff}_{k}\left( M,\mathbb{F}\right) \times
C^{\infty }\left( M\right) \rightarrow \mathrm{Diff}_{k}\left(
M\right) , \\
\Theta _{k} &:&\left( A,f\right) \rightarrow A_{f},
\end{eqnarray*}%
and 
\begin{eqnarray*}
\widetilde{\Theta }_{k} &:&\mathrm{Diff}_{k}\left( M,\mathbb{F}%
\right) \rightarrow \mathrm{diff}_{k}\left( M\right) , \\
\widetilde{\Theta }_{k} &:&A\rightarrow A_{w},
\end{eqnarray*}%
are $\mathcal{G}\left( M\right) -$invariants, in the sence that 
\begin{eqnarray*}
\phi _{\ast }\left( A_{f}\right) &=&\phi _{\ast }^{\left( 0\right) }\left(
A\right) _{\phi _{\ast }\left( f\right) }, \\
\phi _{\ast }\left( A_{w}\right) &=&\phi _{\ast }^{\left( 0\right) }\left(
A\right) _{w}.
\end{eqnarray*}%
Therefore, elements of the field $\mathbb{Q}\left( \widetilde{\varkappa }%
_{l}\right) \otimes \mathbb{F},$ that are $\mathcal{G}\left( M\right) -$%
invariants, we call natural $l$-invariants of $\mathbb{F}$-nonlinear
differential operators.

These invariants form the field $\mathbf{F}_{l}$.

Consider two vector bundles over $J^{0}\left( M\right) :$ 
\begin{eqnarray*}
\varkappa _{l} &:&J^{l}\left( \widetilde{\psi }_{k}\right) \rightarrow
J^{0}\left( M\right) , \\
\pi _{l,0} &:&J^{l}\left( M\right) \rightarrow J^{0}\left( M\right) ,
\end{eqnarray*}%
and their Whitney sum $\xi _{l,r}=\varkappa _{l}\oplus \pi _{r,0}$.

We call it as the \textit{bundle of related pairs}. The elements of the
total space $R_{l,r}$ of this bundle are pairs $\left( [A]_{\left(
x,y\right) }^{l},[f]_{x}^{r}\right) $ consisting of $l$-jet $[A]_{\left(
x,y\right) }^{l}$ of operator $A\in \mathrm{Diff}_{k}\left( M,%
\mathbb{F}\right) $ at point $\left( x,y\right) \in J^{0}\left( M\right)
=M\times \mathbb{R},$ and $r$-jet $[f]_{x}^{r}$ of a function $f\in
C^{\infty }\left( M\right) $ such that $f\left( x\right) =y.$

Functions of $\mathbb{Q}\left( \xi _{l,r}\right) \otimes \mathbb{F}$, where $%
\mathbb{Q}\left( \xi _{l,r}\right) $ consist of functions rational along
fibres $\pi _{0}\circ \xi _{l,r}$,  that are $\mathcal{G}\left( M\right) -$%
invariants, we call invariants of related pairs.

Their values on the fibre $\pi _{0}\circ \xi _{l,r}$ at the point $b\in M$, form a field of $\mathbf{F}_{l,r}\subset \mathbf{Q}_{l,r}$ of invariants of related pairs.

We have natural embeddings $\mathbf{F}_{l,r}\subset \mathbf{F}_{l^{\prime
}, r^{\prime}}$ and $\mathbf{Q}_{l,r}\subset \mathbf{Q}_{l^{\prime }, r^{\prime }}$, 
when $l<l^{\prime }, r<r^{\prime }$.

Denote by $\mathcal{F}\left( r\right) $ and $\mathcal{Q}\left( r\right) $
their inductive limits in $l.$

Remark, that these fields have $\mathcal{G}\left( M\right) -$invariant
derivation 
\begin{equation*}
\nabla =\frac{d}{du}\otimes \partial _{u}^{\left( r\right) }.
\end{equation*}%
Namely, define the following%
\begin{equation*}
\nabla \left( fg\right) =\frac{df}{du}g+f\partial _{u}^{\left( r\right)
}\left( g\right) ,
\end{equation*}%
where $f$ function on $J^{N}\left( \widetilde{\psi }_{k}\right) $ and $g$ on 
$J^{r}\left( M\right) ,$ and $\partial _{u}^{\left( r\right) }$ is the $r$
-th prolongation of vector field $\partial _{u}$ in $J^{r}\left( M\right) .$

This derivation is $\mathcal{G}\left( M\right)-$invariant because $u$ and $%
\partial_{u}$ are $\mathcal{G}\left( M\right)-$in\-variant.

\subsection{Construction of natural invariants}

At first, we remark that $\mathcal{F}\left( 0\right) -$ is the field of
natural invariants of of $\mathbb{F}$-nonlinear differential operators.

Secondly, as we have seen, the maps 
\begin{eqnarray*}
\varrho _{l} &:&R_{l, l}\rightarrow J^{l}\left( \psi _{k}\right) , \\
\rho _{l} &:&\left( [A]_{\left( x, y\right) }^{l},[f]_{x}^{l}\right)
\rightarrow \lbrack A_{f}]_{x}^{l},
\end{eqnarray*}%
are $\mathcal{G}\left( M\right) -$invariant.

Therefore, for any natural $l$-invariant $I$ of linear differential
operators we get invariant $\rho _{l}^{\ast }\left( I\right) $ that is the
natural invariant of related pairs.

To get invariants of nonlinear $\mathbb{F}$\textit{-} differential operators
from the invariants of related pairs we will apply the descent procedure
that was discussed in our previous papers \cite{LY2}.

Namely, let $I_{0}\in \mathcal{F}\left( r\right) $ be an invariant of
related pairs and let $I_{1}=\nabla \left( I_{0}\right) , \ldots ,I_{N}=\nabla
^{N}\left( I_{0}\right) $ be its invariant derivatives. Remark that the
transcendence degree of the field $\mathcal{Q}\left( r\right) $over $%
\mathcal{Q}\left( 0\right) $ equals $\dim \pi _{r,0}$. Therefore, if $N\geq\dim\pi _{r,0}$, then there must be polynomial relations between the invariants $I_{0},\ldots, I_{N}$.  

Denote by $\mathcal{J}\subset \mathcal{Q}\left( 0\right) [X_{0},...,X_{N}]$
the ideal of these relations.

\begin{theorem}
Let $b_{1},...,b\,_{K}$ be the reduced Groebner basis in the ideal $\mathcal{%
J}$ with respect to the standard lexicographic order. Then the coefficients
of polynomials $b_{i}\in \mathcal{Q}\left( 0\right) $ are natural invariants
of $\mathbb{F}$- nonlinear operators.
\end{theorem}

\begin{proof}
The proof repeats the proof of the similar theorem in \cite{LY2}. \newline
Namely, by the construction, the diffeomorphism pseudogroup $\mathcal{G}%
\left( M\right) $ preserves the ideal $\mathcal{J}$ and the lexicographic
order.\newline
Remark also, that a reduced Groebner basis in an ideal with respect to the
lexicographic order is unique \cite{Cox}. Therefore, the $\mathcal{G}\left(
M\right) $-action preserves elements of the basis as well as their
coefficients.
\end{proof}

\section{Natural equivalence of $\mathbb{F}$-nonlinear operators}

First of all we remark that theorem on n-invariants could be applied to
operators $A\in \mathrm{Diff}_{k}\left( M,\mathbb{F}\right) .$

Thus, let $I_{0}=u,I_{1},...,I_{n}$ be invariants in general position and
such that 
\begin{eqnarray*}
\widehat{d}I_{0}\wedge \cdots \wedge \widehat{d}I_{n} &\neq &0. \\
\widehat{d}I_{1}\wedge \cdots \wedge \widehat{d}I_{n} &\neq &0,
\end{eqnarray*}%
and $u,I_{1}\left( A\right) ,...,I_{n}\left( A\right) $ are coordinates in a
domain $\mathcal{O\subset }J^{0}\left( M\right) .$

Then, we can substitute invariants $I_{i},$ $i>0,$ by new invariants $%
I_{i}^{\prime }=F_{i}\left( u,I_{1},...,I_{n}\right) $ such that 
\begin{equation}
\frac{dI_{i}^{\prime }\left( A\right) }{du}=0,  \label{adj ext}
\end{equation}%
for $i=1,..,n.$

In this situation, we call a triple $\left( A,\mathcal{O},\left(
u,I_{1},...,I_{n}\right) \right) $ \textit{adjusted in domain }$\mathcal{%
O\subset }J^{0}\left( M\right) $ if functions $\left( u,I_{1}\left( A\right)
,..,I_{n}\left( A\right) \right) $ are coordinates in the domain $\mathcal{O}%
,$ satisfied the above conditions (\ref{adj ext}).

Consider the extanded space 
$\widetilde{\Phi} _{k}=\mathbb{R}^{n+1}\times 
\mathbb{R}^{\binom{n+k}{k}}$ with coordinates $\left(
y_{0},y_{1},...,y_{n},Y_{\alpha },\left\vert \alpha \right\vert \leq
k\right) .$

Let $\widetilde{\square }:C^{\infty }\left( J^{\infty }\left( \widetilde{%
\psi }_{k}\right) \right) \rightarrow C^{\infty }\left( J^{\infty +k}\left( 
\widetilde{\psi }_{k}\right) \right) $ the universal operator, associated
with operators $\mathrm{Diff}_{k}\left( M,\mathbb{F}\right)$, and 
\begin{eqnarray*}
\widetilde{\phi }_{A} &:&\mathcal{O}\subset J^{0}\left( M\right)
\rightarrow\widetilde{\Phi}_{k}, \\
\widetilde{\phi }_{A} &:&\left( x,u\right) \in \mathcal{O}\rightarrow \left(
y_{0}=u,y_{1}=I_{1}\left( A\right), \ldots, y_{n}=I_{n}\left( A\right), Y_{\alpha }=\widetilde{J}_{\alpha }\left( A\right) \right),
\end{eqnarray*}%
where 
\begin{equation*}
\widetilde{J}_{\alpha }=\frac{1}{\alpha !}\widetilde{\square }\left(
I^{\alpha }\right) ,
\end{equation*}%
and $\alpha =\left( \alpha _{1},..,\alpha _{n}\right) $ are multi indices of
lenght $\leq k.$

Then, the $\left( n+1\right) $-dimentional submanifold $\widetilde{\Sigma }%
_{A}=\widetilde{\phi }_{A}\left( \mathcal{O}\right) \subset \widetilde{%
\Phi}_{k}$ we call\textit{\ model of the} $\mathbb{F}$ -\textit{%
nonlinear operator} $A$ in the domain $\mathcal{O\subset }J^{0}\left(
M\right) $, and, as above, \ we get the following result.

\begin{theorem}
Let $\left( A_{1},\mathcal{O}_{1},\left( u,I_{1},...,I_{n}\right) \right) $
and $\left( A_{2},\mathcal{O}_{2},\left( u,I_{1},...,I_{n}\right) \right) $
, $A_{i}\in \mathrm{Diff}_{k}\left( M,\mathbb{F}\right) ,\mathcal{O}%
_{i}\subset J^{0}\left( M\right) ,$ are adjusted triples. Then there exist a
diffeomorphism $\phi \in \mathcal{G}\left( M\right) $ such that $\phi
^{\left( 0\right) }:\mathcal{O}_{1}\rightarrow \mathcal{O}_{2},$ and $\phi
_{\ast }^{\left( 0\right) }\left( A_{1}\right) =A_{2}$ if and only if their
models are equal, $\widetilde{\Sigma }_{A_{1}}=$ $\widetilde{\Sigma }%
_{A_{2}}.$
\end{theorem}
\bigskip


\end{document}